\documentclass[11pt]{article}

\usepackage[T1]{fontenc}
\usepackage[utf8]{inputenc}
\usepackage{lmodern}
\usepackage[margin=1in]{geometry}
\usepackage{amsmath,amssymb,amsthm,mathtools}
\usepackage{booktabs}
\usepackage{array}
\usepackage{enumitem}
\usepackage{xcolor}
\definecolor{linkblue}{rgb}{0.0,0.12,0.45}
\usepackage[colorlinks=true,linkcolor=linkblue,citecolor=linkblue,urlcolor=linkblue]{hyperref}
\usepackage[nameinlink,capitalise,noabbrev]{cleveref}
\usepackage{microtype}

\hypersetup{
  pdftitle={Morphogenesis Across n: Overlays, Emergence Thresholds, and Weak Self-Similarity in the Partition Graph},
  pdfauthor={Fedor B. Lyudogovskiy}
}

\newtheorem{theorem}{Theorem}[section]
\newtheorem{proposition}[theorem]{Proposition}
\newtheorem{corollary}[theorem]{Corollary}
\newtheorem{lemma}[theorem]{Lemma}
\theoremstyle{definition}
\newtheorem{definition}[theorem]{Definition}
\newtheorem{problem}[theorem]{Problem}
\theoremstyle{remark}
\newtheorem{remark}[theorem]{Remark}

\title{Morphogenesis Across $n$: Overlays, Emergence Thresholds, and Weak Self-Similarity in the Partition Graph}
\author{Fedor B. Lyudogovskiy}
\date{}

\begin{document}
\maketitle

\begin{abstract}
We study the partition graphs $G_n$ as a growing family of discrete geometric objects and develop a formal framework for comparing their structures across different levels. The main tool is a family of Ferrers-translation maps
\[
T_\tau:G_n\to G_{n+k},
\qquad (T_\tau(\lambda))'=\lambda'+\tau',
\]
defined for fixed partitions $\tau\vdash k$. We prove that these maps are induced graph embeddings, which gives a rigorous notion of translation overlay: an induced copy of $G_n$ inside $G_{n+k}$. As a consequence, every finite rooted induced motif persists to all higher levels under translation overlays, and overlay-monotone finitely witnessed properties admit stable emergence thresholds.

We then apply this framework to the growth of local complexity. In particular, we show that the maximum degree, the maximum local clique number, and the maximum local simplex dimension are nondecreasing in $n$, so that lower-bound thresholds for these extremal quantities are automatically stable. We also introduce a small canonical family of motifs drawn from several structural zones of the partition graph family: two boundary-corner motifs and three theorem-safe weak rooted templates representing axial, rectangular rear-root, and square-based rear-contour behavior. For these objects we obtain rigorous occurrence statements, explicit upper bounds on first appearance, and stable-threshold statements.

Alongside the theorem-bearing threshold theory, we distinguish a broader atlas layer of first observed occurrences: richer axial bridges, longer spine fragments, nonminimal rear motifs, and other repeated patterns not yet covered by strict transport results. This yields a conservative framework for discussing morphogenesis across $n$, separating exact structural persistence from heuristic resemblance and supporting a weak notion of self-similarity based on recurring finite motifs and repeated local fragments.
\end{abstract}

\medskip
\noindent\textbf{Keywords.} partition graph, integer partitions, Ferrers diagrams, overlays, emergence thresholds, weak self-similarity.

\medskip
\noindent\textbf{MSC 2020.} 05C25, 05A17, 05C30.

\section{Introduction}

The partition graph $G_n$ is the graph whose vertices are the integer partitions of $n$, with two partitions adjacent when one is obtained from the other by a single elementary transfer of one unit between parts, followed by reordering. Equivalently, one may first append trailing zero parts if necessary, perform one transfer between two parts, and then delete any zero parts and reorder. In particular, one elementary transfer may create a new part of size $1$ or delete a part by reducing it to $0$. This transfer-based viewpoint is related to minimal-change constructions for integer partitions and to the Gray-code literature on partition families \cite{Savage1989,RasmussenSavageWest1995,Mutze2023}; for a different partition-based graph model, see \cite{Bal2022}. In the preceding papers of this series, the family $(G_n)$ has been studied from several complementary viewpoints, including local vertex morphology, boundary and rear structure, axial organization, simplex stratification, directional anisotropy, the interpretation of $G_n$ as a growing discrete geometric object, and the topological background supplied by the associated clique complexes \cite{Lyudogovskiy2026LocalMorphology,Lyudogovskiy2026GrowingObject,Lyudogovskiy2026AxialMorphology,Lyudogovskiy2026SimplexStratification,Lyudogovskiy2026DegreeLandscape,Lyudogovskiy2026BoundaryRear,Lyudogovskiy2026DirectionalGeometry,Lyudogovskiy2026CliqueComplex}. The present paper shifts the emphasis from the internal structure of a fixed graph $G_n$ to the dynamics of the family
\[
G_1,G_2,G_3,\dots
\]
as a coherent morphogenetic sequence. For standard background on integer partitions and Ferrers diagrams, see \cite{Andrews1998}.

Our main goal is to introduce a strict language for comparing different levels of this family and for distinguishing rigorously between exact structural transport and merely visual resemblance. At first sight, the graphs $G_n$ often appear to grow in a recognizable way: certain local motifs recur, parts of the boundary framework look stable, and new layers of complexity seem to emerge by threshold events rather than by arbitrary rearrangement. However, such impressions are not, by themselves, mathematically meaningful. The first task of the present paper is therefore to replace the informal notion of an ``overlay'' between $G_n$ and $G_{n+k}$ by an explicit family of induced graph embeddings.

The basic mechanism is given by Ferrers translations. For every fixed partition $\tau\vdash k$, we consider the map
\[
T_\tau:G_n\to G_{n+k},
\qquad
(T_\tau(\lambda))'=\lambda'+\tau',
\]
defined in terms of conjugate partitions. We prove that these maps are induced graph embeddings. This yields a strict and reusable notion of translation overlay, and with it a rigorous persistence principle for finite induced motifs: once a finite rooted configuration occurs at some level, all its translated copies persist at all higher levels. In particular, the overlay formalism provides a clean threshold language for finitely witnessed properties and a direct monotonicity mechanism for extremal local invariants.

A second aim of the paper is to clarify which first-appearance phenomena belong to theorem-bearing threshold theory and which remain computational. Some thresholds are strict: they are attached to overlay-monotone properties, or to lower bounds on extremal quantities such as the maximum degree, the maximum local clique number, and the maximum local simplex dimension. Other first appearances are only empirical at present: richer axial bridge types, longer spine fragments, more elaborate rear motifs, exact spectrum phenomena, or broader center-versus-boundary regime changes. To keep this distinction explicit, we separate \emph{strict thresholds} from \emph{atlas thresholds}, the latter meaning first observed occurrences in the computed range without a general persistence theorem.

The theorem-bearing part of the paper is intentionally modest and local. Rather than attempting to transport the entire framework, spine, or central region by a single global rule, we work with a small canonical family of motifs. On the strict side, this includes two canonical boundary-corner motifs, denoted $BL1$ and $BR1$, together with three theorem-safe weak rooted templates $P2^\circ$, $E1^\circ$, and $R_{sq}^\circ$, extracted from the first nontrivial axial, rectangular, and square-based occurrences in the atlas. The role of these objects is not to exhaust the morphology of the partition graph family, but to provide representative threshold carriers from several distinct structural zones: boundary, axial, rear-root, and square-based rear contour.

The paper also records a compact atlas framework for first appearances and repeated patterns across $n$, together with a canonical registry of the main motifs used here. Its role is comparative as well as illustrative. Once a strict overlay language is in place, the atlas records where new local motifs first appear, where repeated fragments suggest weak self-similarity, and where complexity seems to migrate between boundary, axis, rear, and interior. It also helps control overstatement: several visually natural patterns turn out to require weaker correspondence language rather than exact transport theorems.

From this perspective, the paper is not a general theory of self-similarity for the family $(G_n)$. The term \emph{weak self-similarity} is used here in a deliberately conservative sense: recurrent finite motifs, repeated local fragments, and persistent normalized structural profiles, rather than exact global replication. Likewise, we do not claim a canonical global transition map from $G_n$ to $G_{n+1}$ that simultaneously respects all distinguished subsystems. What we do claim is that there exists a strict overlay mechanism rich enough to support a threshold theory for finite motifs and local extremal complexity, together with an atlas framework broad enough to organize the observed morphogenesis across levels without overstating what has been computed explicitly here.

The main contributions of the paper may be summarized as follows.

\begin{enumerate}[leftmargin=2em]
    \item We introduce Ferrers-translation overlays $T_\tau:G_n\to G_{n+k}$ and prove that they are induced graph embeddings.
    \item We derive a persistence principle for finite rooted induced motifs and a strict threshold formalism for overlay-monotone finitely witnessed properties.
    \item We prove monotonicity of the extremal local quantities $\Delta_n$, $\Omega_n$, and $S_n$, yielding stable thresholds for lower bounds on degree and local simplex complexity.
    \item We apply this formalism to a canonical motif package consisting of $BL1$, $BR1$, $P2^\circ$, $E1^\circ$, and $R_{sq}^\circ$, thereby obtaining representative threshold statements from several morphologically distinct regions of the graph family.
    \item We distinguish theorem-level thresholds from atlas thresholds and record a compact atlas framework for richer motifs and repeated patterns that are not yet covered by the strict transport theory.
\end{enumerate}

The paper is organized as follows. Section~\ref{sec:growth-language} introduces the growth language, the translation-overlay maps, and their basic consequences. Section~\ref{sec:persistence} develops the persistence formalism and the transport of finite structure. Section~\ref{sec:thresholds} establishes the threshold framework, proves stable-threshold results for extremal local complexity, and applies the theory to the canonical motifs. Section~\ref{sec:comparative-growth} turns to comparative growth, repeated patterns, and weak self-similarity. Section~\ref{sec:atlas} records the current atlas framework, canonical registry, and comparative motif chronology. The final section collects open problems concerning stronger transport mechanisms, broader classes of persistent motifs, and the possible relation between local complexity growth and global geometric simplification.

\section{Growth language and translation overlays}
\label{sec:growth-language}

This section introduces a strict language for comparing different levels of the partition graph family
\[
G_1,G_2,G_3,\dots
\]
and to separate exact structural transport from looser forms of comparison. The basic mechanism is given by Ferrers translations in conjugate coordinates. These maps will provide a canonical class of induced embeddings
\[
G_n \hookrightarrow G_{n+k},
\]
which we use as the formal core of the overlay language in the present paper.

\subsection{Ferrers translations and growth maps}
\label{subsec:ferrers-translations}

Let $\lambda\vdash n$ be a partition, and write
\[
\lambda'=(\lambda'_1,\lambda'_2,\dots)
\]
for its conjugate partition. As usual, we allow trailing zeros when convenient.

For a fixed partition $\tau\vdash k$, define the map
\[
T_\tau:G_n\to G_{n+k}
\]
by
\[
(T_\tau(\lambda))'=\lambda'+\tau'.
\]
Since $\lambda'$ and $\tau'$ are weakly decreasing sequences, their coordinatewise sum $\lambda'+\tau'$ is again weakly decreasing, and its total size is $n+k$. Hence it is the conjugate of a unique partition of $n+k$, so $T_\tau$ is well defined.
Equivalently, in Ferrers-diagram language, $T_\tau(\lambda)$ is obtained by adding the Ferrers diagram of $\tau$ to that of $\lambda$ in conjugate coordinates.

\begin{definition}
\label{def:ferrers-translation}
For $\tau\vdash k$, the map $T_\tau:G_n\to G_{n+k}$ defined by
\[
(T_\tau(\lambda))'=\lambda'+\tau'
\]
is called the \emph{Ferrers translation of type $\tau$}.
\end{definition}

Two special families will be used repeatedly.

\begin{definition}
\label{def:row-column-growth}
For $k\ge 1$, define
\[
R_k:=T_{(k)},\qquad C_k:=T_{(1^k)}.
\]
We call $R_k$ the \emph{row-growth overlay} of length $k$, and $C_k$ the \emph{column-growth overlay} of length $k$.
\end{definition}

These two families reflect the directional anisotropy of the partition graph family: some distinguished structures are naturally transported by row growth, others by column growth, and no single one-step map should be expected to govern all structural zones simultaneously; compare \cite{Lyudogovskiy2026DirectionalGeometry,Lyudogovskiy2026BoundaryRear}.

\subsection{Strict overlays, partial overlays, and transition correspondences}
\label{subsec:overlay-language}

We now distinguish three levels of comparison across different values of $n$.

\begin{definition}
\label{def:strict-overlay}
A \emph{strict overlay} of $G_n$ in $G_{n+k}$ is an induced subgraph of $G_{n+k}$ of the form
\[
T_\tau(G_n)
\]
for some partition $\tau\vdash k$.
\end{definition}

Thus a strict overlay is not merely a visual superposition: it is an induced embedded copy of the whole graph $G_n$.

\begin{definition}
\label{def:partial-overlay}
Let $H\subseteq G_n$ be a subgraph. A \emph{partial overlay} of $H$ in $G_{n+k}$ is an induced copy of $H$ of the form
\[
T_\tau(H)\subseteq G_{n+k}
\]
for some partition $\tau\vdash k$.
\end{definition}

This is the appropriate notion when one transports only a selected subsystem, such as a finite motif, a local neighborhood pattern, or a chosen fragment of a distinguished structural zone.

\begin{definition}
\label{def:transition-correspondence}
A \emph{transition correspondence}
\[
\Phi_{n\to n+k}:G_n\rightsquigarrow G_{n+k}
\]
is any explicitly specified rule that associates selected vertices or selected substructures of $G_n$ with corresponding vertices or substructures of $G_{n+k}$, without requiring inducedness or even a globally defined map on all vertices.
\end{definition}

Transition correspondences are useful for broader morphological comparison, especially when discussing the framework, the self-conjugate axis, the spine, the rear region, or other distinguished subsystems for which no single strict overlay is either available or natural.

\begin{remark}
\label{rem:overlay-vs-correspondence}
The terminology above is intentionally strict. In particular, the word \emph{overlay} will always refer to an induced transport mechanism coming from a Ferrers translation, whereas looser comparisons will be described as \emph{transition correspondences}, \emph{transport heuristics}, or \emph{atlas-level similarities}.
\end{remark}

\subsection{Adjacency in conjugate coordinates}
\label{subsec:adjacency-conjugate}

The key point is that adjacency in $G_n$ admits a simple description in conjugate coordinates.

\begin{lemma}
\label{lem:conjugation-automorphism}
For each $n$, the map
\[
\lambda\mapsto \lambda'
\]
is a graph automorphism of $G_n$.
\end{lemma}

\begin{proof}
If $\mu$ is obtained from $\lambda$ by an elementary unit transfer, then in Ferrers-diagram language one moves a single terminal cell from one row of $\lambda$ to another, allowing the target row to be initially empty and deleting the source row if it becomes empty after the move, followed by reordering. After transposing Ferrers diagrams, rows and columns are interchanged, so the same move becomes the transfer of a single terminal cell between two columns. This is again an elementary unit transfer between parts, now for the conjugate partitions $\lambda'$ and $\mu'$. Hence adjacency is preserved by conjugation.

Since conjugation is an involution, the same argument applies in reverse. Therefore $\lambda\mapsto\lambda'$ is a graph automorphism of $G_n$.
\end{proof}

\begin{lemma}
\label{lem:adjacency-conjugate-l1}
Let $\lambda,\mu\vdash n$. Then the following are equivalent:
\begin{enumerate}[leftmargin=2em]
    \item $\lambda$ and $\mu$ are adjacent in $G_n$;
    \item $\lambda'$ and $\mu'$ differ by one elementary unit transfer between parts;
    \item after padding by trailing zeros if necessary,
    \[
    \|\lambda'-\mu'\|_1=2.
    \]
\end{enumerate}
\end{lemma}

\begin{proof}
The equivalence of (1) and (2) is exactly \Cref{lem:conjugation-automorphism}.

If $\mu'$ is obtained from $\lambda'$ by decreasing one part by $1$ and increasing another part by $1$, then the difference vector $\mu'-\lambda'$ has exactly one coordinate equal to $+1$, exactly one coordinate equal to $-1$, and all other coordinates equal to $0$. Hence
\[
\|\lambda'-\mu'\|_1=2.
\]

Conversely, if $\|\lambda'-\mu'\|_1=2$, then since $\lambda'$ and $\mu'$ have the same total sum $n$, the difference vector $\mu'-\lambda'$ must again consist of exactly one $+1$, exactly one $-1$, and zeros elsewhere. Thus $\mu'$ is obtained from $\lambda'$ by one elementary unit transfer between parts. This gives (2), and hence (1).
\end{proof}

\begin{lemma}
\label{lem:translation-preserves-l1}
For every fixed partition $\tau$ and every $\lambda,\mu\vdash n$,
\[
\|(T_\tau(\lambda))'-(T_\tau(\mu))'\|_1
=
\|\lambda'-\mu'\|_1.
\]
\end{lemma}

\begin{proof}
By definition,
\[
(T_\tau(\lambda))'=\lambda'+\tau',
\qquad
(T_\tau(\mu))'=\mu'+\tau'.
\]
Therefore
\[
(T_\tau(\lambda))'-(T_\tau(\mu))'=\lambda'-\mu',
\]
and the claim follows immediately.
\end{proof}

\subsection{The Translation Overlay Theorem}
\label{subsec:translation-overlay-theorem}

We now obtain the main structural result of the paper.

\begin{theorem}[Translation Overlay Theorem]
\label{thm:translation-overlay}
Let $\tau\vdash k$. Then
\[
T_\tau:G_n\longrightarrow G_{n+k}
\]
is an injective induced graph embedding.

Equivalently, $T_\tau(G_n)$ is a strict overlay of $G_n$ inside $G_{n+k}$.
\end{theorem}

\begin{proof}
Injectivity is immediate: if $T_\tau(\lambda)=T_\tau(\mu)$, then
\[
\lambda'+\tau'=\mu'+\tau',
\]
hence $\lambda'=\mu'$, and therefore $\lambda=\mu$.

It remains to show that adjacency is both preserved and reflected. By Lemma~\ref{lem:adjacency-conjugate-l1},
\[
\lambda\sim\mu
\quad\Longleftrightarrow\quad
\|\lambda'-\mu'\|_1=2.
\]
By Lemma~\ref{lem:translation-preserves-l1},
\[
\|\lambda'-\mu'\|_1
=
\|(T_\tau(\lambda))'-(T_\tau(\mu))'\|_1.
\]
Applying Lemma~\ref{lem:adjacency-conjugate-l1} again, now at level $n+k$, we obtain
\[
\lambda\sim\mu
\quad\Longleftrightarrow\quad
T_\tau(\lambda)\sim T_\tau(\mu).
\]
Thus $T_\tau$ preserves both adjacency and non-adjacency, so it is an induced graph embedding.
\end{proof}

\begin{corollary}
\label{cor:finite-pattern-persistence}
Let $H$ be a finite induced subgraph of $G_n$. Then for every partition $\tau$, the image $T_\tau(H)$ is an induced subgraph of $G_{n+|\tau|}$ isomorphic to $H$.

In particular, every finite rooted induced motif occurring at some level persists at infinitely many higher levels.
\end{corollary}

\begin{proof}
Restrict the induced embedding $T_\tau:G_n\hookrightarrow G_{n+|\tau|}$ to the chosen induced copy of $H$.
\end{proof}

\begin{remark}
\label{rem:exact-transport-limitations}
Theorem~\ref{thm:translation-overlay} provides an exact transport mechanism for finite induced structure. It does not imply that every morphologically distinguished subsystem of $G_n$ is transported canonically in the strong geometric sense. In later sections, this is why we distinguish carefully between theorem-safe rooted templates and stronger typed interpretations used only in the atlas and the morphological discussion.
\end{remark}

\subsection{Local consequences and monotonicity of extremal complexity}
\label{subsec:local-monotonicity}

The overlay mechanism immediately implies that old local configurations cannot disappear under growth, even though new neighbors or new cliques may arise outside the overlay image.

Let $\deg(\lambda)$ denote the degree of $\lambda$ in $G_n$. Let $\omega_{\mathrm{loc}}(\lambda)$ be the maximum size of a clique in $G_n$ containing $\lambda$, and let
\[
s_{\mathrm{loc}}(\lambda)=\omega_{\mathrm{loc}}(\lambda)-1
\]
be the local simplex dimension at $\lambda$.

Define the extremal quantities
\[
\Delta_n=\max_{\lambda\vdash n}\deg(\lambda),\qquad
\Omega_n=\max_{\lambda\vdash n}\omega_{\mathrm{loc}}(\lambda),\qquad
S_n=\max_{\lambda\vdash n}s_{\mathrm{loc}}(\lambda).
\]

\begin{lemma}
\label{lem:local-quantities-nondecrease}
For every partition $\tau$ and every vertex $\lambda\vdash n$,
\[
\deg(T_\tau(\lambda))\ge \deg(\lambda),
\]
\[
\omega_{\mathrm{loc}}(T_\tau(\lambda))\ge \omega_{\mathrm{loc}}(\lambda),
\]
and therefore
\[
s_{\mathrm{loc}}(T_\tau(\lambda))\ge s_{\mathrm{loc}}(\lambda).
\]
\end{lemma}

\begin{proof}
By Theorem~\ref{thm:translation-overlay}, the map $T_\tau$ embeds $G_n$ as an induced subgraph of $G_{n+|\tau|}$. Hence every neighbor of $\lambda$ is carried to a distinct neighbor of $T_\tau(\lambda)$, which gives
\[
\deg(T_\tau(\lambda))\ge \deg(\lambda).
\]

Likewise, if $K$ is a clique containing $\lambda$, then $T_\tau(K)$ is a clique containing $T_\tau(\lambda)$. Hence
\[
\omega_{\mathrm{loc}}(T_\tau(\lambda))\ge \omega_{\mathrm{loc}}(\lambda).
\]
The statement for $s_{\mathrm{loc}}$ follows immediately from the relation
\[
s_{\mathrm{loc}}=\omega_{\mathrm{loc}}-1.
\]
\end{proof}

\begin{proposition}
\label{prop:extremal-local-monotonicity}
The sequences
\[
(\Delta_n)_{n\ge 1},\qquad
(\Omega_n)_{n\ge 1},\qquad
(S_n)_{n\ge 1}
\]
are nondecreasing in $n$. Equivalently, for all $n\ge 1$ and all $k\ge 1$,
\[
\Delta_{n+k}\ge \Delta_n,\qquad
\Omega_{n+k}\ge \Omega_n,\qquad
S_{n+k}\ge S_n.
\]
\end{proposition}

\begin{proof}
Fix $n$ and $k$, and choose any partition $\tau\vdash k$.

Let $\lambda\vdash n$ realize $\Delta_n$. By Lemma~\ref{lem:local-quantities-nondecrease},
\[
\deg(T_\tau(\lambda))\ge \deg(\lambda)=\Delta_n.
\]
Since $T_\tau(\lambda)$ is a vertex of $G_{n+k}$, it follows that
\[
\Delta_{n+k}\ge \Delta_n.
\]

The same argument applies to $\Omega_n$ and $S_n$: if $\lambda$ realizes either extremum at level $n$, then its translated image realizes at least the same value at level $n+k$. Hence
\[
\Omega_{n+k}\ge \Omega_n,\qquad S_{n+k}\ge S_n.
\]
\end{proof}

\begin{remark}
\label{rem:local-extrema-vs-full-spectra}
Proposition~\ref{prop:extremal-local-monotonicity} concerns only extremal local quantities. It does not imply monotonicity of the full degree spectrum, of the set of realized local simplex dimensions, or of any exact-value distributional profile. Those finer phenomena will be treated later as atlas-level first-appearance data rather than as automatic consequences of the overlay mechanism.
\end{remark}

\subsection{Directional overlays and anisotropic transport}
\label{subsec:directional-overlays}

Although Theorem~\ref{thm:translation-overlay} is completely general, certain special overlay families are better adapted to specific morphological zones.

The row-growth family $R_k=T_{(k)}$ naturally adds horizontal mass in Ferrers coordinates, while the column-growth family $C_k=T_{(1^k)}$ adds vertical mass. Accordingly, these two families interact differently with the distinguished front and boundary structures already identified in the earlier papers of the series \cite{Lyudogovskiy2026BoundaryRear,Lyudogovskiy2026DirectionalGeometry}.

At the most elementary level, the canonical left boundary-corner motif introduced in Section~\ref{sec:thresholds} is compatible with the row-growth family, whereas its right-hand conjugate counterpart is compatible with the column-growth family. This already shows that the boundary framework should not be expected to admit a single preferred one-step transport rule. Instead, the framework is better viewed as directionally decomposed, with different canonical overlay families acting naturally on different sides.

More generally, the row/column dichotomy provides a strict background for the anisotropy language developed elsewhere in the project \cite{Lyudogovskiy2026DirectionalGeometry}. Some transport phenomena are exact and arise from a chosen translation family; others are only correspondence-level and must be described more cautiously. In this sense, directional overlays form the strict core of the comparative growth language, while broader morphogenetic claims remain partly computational and partly heuristic.

\begin{remark}
\label{rem:section2-output}
Only two outputs of this section are used later: the strict overlay mechanism $T_\tau$ and the induced monotonicity of extremal local complexity.
\end{remark}

\section{Persistence and transport of finite structure}
\label{sec:persistence}

Beyond giving exact induced embeddings, the overlay mechanism yields a persistence theory for finite structure across the growing partition graph family. This section isolates that finite-witness mechanism, clarifies the hierarchy of transport notions used throughout the paper, and prepares the threshold applications of Section~\ref{sec:thresholds}.

\subsection{Overlay-monotone properties and finite witnesses}
\label{subsec:overlay-monotone-properties}

The most robust consequences of the translation-overlay theorem concern properties that admit finite witnesses and are preserved under strict overlays.

\begin{definition}
\label{def:overlay-monotone}
A property $\mathcal P$ of the level graphs $G_n$ is called \emph{overlay-monotone} if for every $n\ge 1$, every $k\ge 1$, and every partition $\tau\vdash k$,
\[
G_n \text{ has } \mathcal P
\quad\Longrightarrow\quad
G_{n+k} \text{ has } \mathcal P.
\]
Equivalently, once $\mathcal P$ is realized at some level, it survives under every strict translation overlay.
\end{definition}

\begin{definition}
\label{def:motif-defined-property}
A property $\mathcal P$ is called \emph{motif-defined} if there exists a class $\mathcal M$ of finite rooted induced graphs such that
\[
G_n \text{ has } \mathcal P
\quad\Longleftrightarrow\quad
G_n \text{ contains an induced rooted copy of some } M\in\mathcal M.
\]
Here an \emph{induced rooted copy} of $M$ means an induced subgraph of $G_n$ together with an isomorphism preserving the designated root data of $M$ (a distinguished vertex, or an ordered distinguished tuple, according to the template).
\end{definition}

The basic example is the property
\[
\mathcal P_M:\qquad
G_n \text{ contains an induced rooted copy of a fixed finite motif } M.
\]

\begin{proposition}
\label{prop:motif-defined-overlay-monotone}
Every motif-defined property is overlay-monotone.
\end{proposition}

\begin{proof}
Suppose that $G_n$ has a motif-defined property $\mathcal P$. Then $G_n$ contains an induced rooted copy of some $M\in\mathcal M$, that is, an induced subgraph equipped with a root-preserving isomorphism to $M$. By Corollary~\ref{cor:finite-pattern-persistence}, every strict overlay $T_\tau$ carries this copy to an induced copy of the same finite graph inside $G_{n+|\tau|}$, and the image of the distinguished root data gives a root-preserving isomorphism to $M$ again. Hence $G_{n+|\tau|}$ also has $\mathcal P$. Thus $\mathcal P$ is overlay-monotone.
\end{proof}

\begin{corollary}
\label{cor:fixed-motif-overlay-monotone}
For every fixed finite rooted induced graph $M$, the property
\[
\mathcal P_M:\qquad
G_n \text{ contains an induced rooted copy of } M
\]
is overlay-monotone.
\end{corollary}

\begin{remark}
\label{rem:finitely-witnessed-scope}
The role of Definitions~\ref{def:overlay-monotone} and \ref{def:motif-defined-property} is structural rather than merely formal. They identify a class of properties for which comparison across levels is exact, in contrast to broader morphological phenomena that will later be treated only at atlas level.
\end{remark}

\subsection{Persistence of rooted motifs and local patterns}
\label{subsec:persistence-rooted-motifs}

Corollary~\ref{cor:finite-pattern-persistence} already gives the rooted version needed later. If $M$ is a fixed finite rooted induced graph and $G_n$ contains an induced rooted copy of $M$, then every translation overlay $T_\tau$ carries this copy to an induced rooted copy of $M$ in $G_{n+|\tau|}$. Thus every fixed rooted induced motif, once realized, persists at all higher levels.

\begin{remark}
\label{rem:rooted-vs-unrooted}
The rooted formulation is often essential even when the underlying graph-theoretic shape is small. The same abstract triangle, for example, may serve as a boundary-corner witness, a rear-root witness, or a simplex witness, depending on the designated root and the intended morphological interpretation.
\end{remark}

This distinction between graph-theoretic shape and morphological interpretation will be used repeatedly in the canonical motif package later in the paper.

\subsection{Partial transport and canonical motif families}
\label{subsec:partial-transport}

The overlay mechanism is most useful when one does not attempt to transport the whole graph $G_n$, but only selected finite subsystems.

\begin{definition}
\label{def:persistent-family}
A family of finite rooted motifs $\mathcal F$ is called \emph{persistent} if for every $M\in\mathcal F$, whenever $M$ occurs at some level $G_n$, it also occurs at every higher level $G_m$, $m\ge n$.
\end{definition}

By the translation-overlay theorem, every family of fixed finite rooted induced motifs is persistent in this sense.

For the present paper, this observation will be applied to a small canonical list of motifs drawn from several structural zones: simplex witnesses, boundary corners, weak axial templates, rear-root templates, and square-based contour templates. The theorem-bearing persistence statements apply only to the rooted templates themselves. Their stronger typed or anchored interpretations will be treated separately in the atlas and morphological discussion.

\begin{remark}
\label{rem:canonical-family-two-layers}
This separation between a theorem-safe rooted template and a stronger morphological interpretation is one of the main structural conventions of the paper.
\end{remark}

\subsection{Strong transport, weak transport, and correspondence}
\label{subsec:strong-weak-transport}

The overlay language naturally leads to a hierarchy of transport notions.

\paragraph{Exact transport.}
This is the strongest case: a finite rooted induced subgraph is carried by a Ferrers translation $T_\tau$ to an induced rooted copy at the higher level. This is the situation covered directly by Theorem~\ref{thm:translation-overlay} and its corollaries.

\paragraph{Partial exact transport.}
Here the transported object is not the whole graph $G_n$, but a selected induced subgraph or rooted motif. This remains fully rigorous and is the main form of transport used later in the paper.

\paragraph{Typed transport.}
In many situations the object being tracked carries additional morphological information: a root may be required to lie on the boundary framework, to be self-conjugate, to be rectangular, or to be square. Such transport may still be exact in special cases, but it does not follow automatically from the overlay theorem and must therefore be checked separately.

\paragraph{Correspondence-level transport.}
At the weakest level, one compares structures across $n$ by an explicit rule, an atlas convention, or a visual alignment, without claiming inducedness or a globally defined embedding. This is the appropriate language for many broader questions concerning the framework, the spine, the rear contour, or the central region beyond the theorem-safe rooted templates.

\begin{remark}
\label{rem:transport-hierarchy}
The paper uses this hierarchy consistently:
\begin{enumerate}[leftmargin=2em]
    \item exact transport for strict overlays and their rooted submotifs;
    \item partial exact transport for finitely witnessed local structure;
    \item typed transport only when separately justified;
    \item correspondence-level comparison for broader morphology not covered by the strict overlay mechanism.
\end{enumerate}
This hierarchy is one of the main safeguards against overinterpreting recurring visual patterns as theorem-level structural persistence.
\end{remark}

\subsection{Directional transport and structural asymmetry}
\label{subsec:directional-transport}

Although every Ferrers translation defines a strict overlay, not every morphological zone is naturally transported by the same translation family. This is already visible in the two basic one-directional families $R_k$ and $C_k$.

On the left side of the graph, the canonical boundary-corner motifs are naturally aligned with row-growth overlays, while on the right side the conjugate motifs are naturally aligned with column-growth overlays. Thus even the front framework is directionally split. This supports the broader anisotropy picture developed elsewhere in the series: transport across levels is often exact, but it is not globally uniform \cite{Lyudogovskiy2026DirectionalGeometry,Lyudogovskiy2026AxialMorphology,Lyudogovskiy2026BoundaryRear}.

The same caution applies more generally to axial and rear structures. A finite rooted template extracted from an axial or rear configuration may persist under all overlays as a graph-theoretic object, while its stronger typed meaning---for example, ``bridge between self-conjugate anchors'' or ``rear-rooted rectangular seed''---may depend on additional conditions not automatically preserved by arbitrary translations.

\begin{remark}
\label{rem:directional-asymmetry}
The partition graph family is therefore not only growing but directionally organized. The strict overlay mechanism is rich enough to support exact transport of finite motifs, yet too coarse to provide a single universal rule for all morphologically distinguished subsystems.
\end{remark}

\subsection{Output of the section}
\label{subsec:persistence-output}

This section isolates the finite-witness mechanism behind persistence and fixes the transport hierarchy used later. For the next section, the essential takeaway is short: finitely witnessed overlay-monotone properties admit a strict threshold theory, whereas broader morphogenetic comparisons generally remain at correspondence or atlas level.

\section{Emergence thresholds and first appearances}
\label{sec:thresholds}

The translation-overlay formalism allows one to distinguish sharply between two notions of first appearance in the partition graph family $(G_n)$. Some properties are carried forward by all strict overlays, or at least by a specified overlay family. For such properties, the first appearance is automatically stable: once the property is realized, it persists at all higher levels allowed by the transport mechanism. By contrast, many global or exact-value morphological features are not known to be overlay-monotone. Their first appearances may still be tracked computationally, but such observations belong to the atlas layer rather than to the strict threshold theory. In this section we make this distinction explicit. We first record the threshold consequences of the general overlay results, then apply them to extremal local complexity and to a small canonical family of motifs used throughout the atlas.

\subsection{Strict thresholds and stable first appearances}
\label{subsec:strict-thresholds}

Recall that a property $\mathcal P$ of the level graphs $G_n$ is called overlay-monotone if
\[
G_n \text{ has } \mathcal P
\quad\Longrightarrow\quad
G_{n+k} \text{ has } \mathcal P
\]
for every $k\ge 1$, or more generally under a specified family of translation overlays.

For such a property, the emergence threshold
\[
n_*(\mathcal P)=\min\{n\ge 1:\; G_n \text{ has } \mathcal P\},
\]
when it exists, coincides with the stable threshold
\[
n_*^{\mathrm{st}}(\mathcal P)=
\min\{n\ge 1:\; G_m \text{ has } \mathcal P \text{ for all } m\ge n\}.
\]

\begin{proposition}
\label{prop:strict-thresholds-general}
Let $\mathcal P$ be an overlay-monotone property. If $n_*(\mathcal P)$ exists, then
\[
n_*^{\mathrm{st}}(\mathcal P)=n_*(\mathcal P).
\]
In particular, every finitely witnessed rooted induced motif has a stable first-appearance threshold.
\end{proposition}

\begin{proof}
Once $\mathcal P$ occurs at the first level $n_*(\mathcal P)$, overlay-monotonicity implies that it holds at every higher level. Thus $n_*^{\mathrm{st}}(\mathcal P)=n_*(\mathcal P)$. The second statement follows from the Translation Overlay Theorem applied to a fixed rooted induced witness.
\end{proof}

\begin{remark}
\label{rem:first-vs-stable-general}
Proposition~\ref{prop:strict-thresholds-general} applies only to properties that are genuinely overlay-monotone, or at least to properties with a transported finite witness. It does not apply automatically to exact-value statements such as ``a vertex of exact degree $d$ occurs'', nor to global shape statements about the entire framework, spine, rear region, or central region.
\end{remark}

\subsection{Thresholds for extremal local complexity}
\label{subsec:extremal-thresholds}

Let
\[
\Delta_n=\max_{\lambda\vdash n}\deg(\lambda),\qquad
\Omega_n=\max_{\lambda\vdash n}\omega_{\mathrm{loc}}(\lambda),\qquad
S_n=\max_{\lambda\vdash n}s_{\mathrm{loc}}(\lambda),
\]
where $\omega_{\mathrm{loc}}(\lambda)$ is the maximum size of a clique containing $\lambda$, and $s_{\mathrm{loc}}(\lambda)=\omega_{\mathrm{loc}}(\lambda)-1$ is the local simplex dimension at $\lambda$. Compare the earlier local and degree-oriented papers in the series \cite{Lyudogovskiy2026LocalMorphology,Lyudogovskiy2026SimplexStratification,Lyudogovskiy2026DegreeLandscape}.

By the monotonicity result proved in the previous section, all three sequences are nondecreasing in $n$. This immediately yields a threshold language for lower bounds on extremal local complexity.

For $d,r,s\ge 0$, define
\[
n_\Delta(d)=\min\{n\ge 1:\; \Delta_n\ge d\},
\]
\[
n_\Omega(r)=\min\{n\ge 1:\; \Omega_n\ge r\},
\]
\[
n_S(s)=\min\{n\ge 1:\; S_n\ge s\},
\]
whenever these sets are nonempty.

\begin{corollary}
\label{cor:extremal-thresholds}
Whenever they exist, the thresholds $n_\Delta(d)$, $n_\Omega(r)$, and $n_S(s)$ are stable. More precisely,
\[
\Delta_n\ge d \text{ for some } n
\quad\Longrightarrow\quad
\Delta_m\ge d \text{ for all } m\ge n,
\]
and likewise
\[
\Omega_n\ge r \text{ for some } n
\quad\Longrightarrow\quad
\Omega_m\ge r \text{ for all } m\ge n,
\]
\[
S_n\ge s \text{ for some } n
\quad\Longrightarrow\quad
S_m\ge s \text{ for all } m\ge n.
\]
Equivalently,
\[
n_*^{\mathrm{st}}(\Delta\ge d)=n_\Delta(d),\qquad
n_*^{\mathrm{st}}(\Omega\ge r)=n_\Omega(r),\qquad
n_*^{\mathrm{st}}(S\ge s)=n_S(s).
\]
\end{corollary}

\begin{proof}
Since the sequences $(\Delta_n)$, $(\Omega_n)$, and $(S_n)$ are nondecreasing, each lower-bound property persists once it is realized.
\end{proof}

\begin{remark}
\label{rem:record-vs-spectrum}
Corollary~\ref{cor:extremal-thresholds} applies to lower-bound thresholds for extremal quantities. It does not imply that the set of realized degree values, the full degree spectrum, or the full set of realized local simplex dimensions is monotone in $n$. Those finer first-appearance questions remain computational in nature.
\end{remark}

\subsection{Canonical motifs and threshold applications}
\label{subsec:canonical-motifs-thresholds}

We now apply the threshold formalism to a small canonical package drawn from boundary, axial, and rear morphology. This package contains two different kinds of objects:
\begin{itemize}[leftmargin=2em]
    \item the fixed rooted templates $P2^\circ$, $E1^\circ$, and $R_{sq}^\circ$;
    \item the level-dependent canonical boundary-corner families $(BL1_n)_{n\ge 4}$ and $(BR1_n)_{n\ge 4}$.
\end{itemize}

For a fixed rooted template $M$, we write
\[
\mathcal P_M(n)
\]
for the property that $G_n$ contains an induced rooted copy of $M$. For the boundary-corner families, the notation $BL1$ or $BR1$ refers to the corresponding canonical realization at level $n$, namely $BL1_n$ or $BR1_n$ as defined below.

\paragraph{Canonical boundary-corner motifs.}
For every $n\ge 4$, define
\[
b_0^{(n)}=(n-1,1),\qquad
b_1^{(n)}=(n-2,1,1),\qquad
b_2^{(n)}=(n-2,2),
\]
and let $BL1_n$ denote the rooted motif on these three vertices, rooted at $b_0^{(n)}$, with ordered secondary vertices $(b_1^{(n)},b_2^{(n)})$. Thus $BL1_n$ is the minimal left boundary-corner motif.

Similarly, define
\[
c_0^{(n)}=(2,1^{n-2}),\qquad
c_1^{(n)}=(3,1^{n-3}),\qquad
c_2^{(n)}=(2,2,1^{n-4}),
\]
and let $BR1_n$ denote the rooted motif on these three vertices, rooted at $c_0^{(n)}$, with ordered secondary vertices $(c_1^{(n)},c_2^{(n)})$. This is the minimal right boundary-corner motif; compare the boundary framework and rear-morphology language in \cite{Lyudogovskiy2026BoundaryRear}.

\begin{proposition}
\label{prop:BL1-BR1-canonical}
For every $n\ge 4$, the three vertices
\[
(n-1,1),\qquad (n-2,1,1),\qquad (n-2,2)
\]
form an induced triangle in $G_n$, and the three vertices
\[
(2,1^{n-2}),\qquad (3,1^{n-3}),\qquad (2,2,1^{n-4})
\]
form an induced triangle in $G_n$.

Equivalently, the canonical motifs $BL1_n$ and $BR1_n$ are well defined for every $n\ge 4$.
\end{proposition}

\begin{proof}
On the left side, the three adjacencies can be written explicitly as follows. Starting from $(n-1,1,0)$, transfer one unit from the first part to the third part to obtain $(n-2,1,1)$; thus $(n-1,1)$ is adjacent to $(n-2,1,1)$. Starting from $(n-1,1)$, transfer one unit from the first part to the second part to obtain $(n-2,2)$; thus $(n-1,1)$ is adjacent to $(n-2,2)$. Finally, starting from $(n-2,1,1)$, transfer one unit from one of the parts equal to $1$ to the other to obtain $(n-2,2,0)$, which reorders to $(n-2,2)$ after deleting the zero part. Hence these three vertices form a triangle.

The right-hand statement follows either by the same direct verification or by conjugation symmetry, since the listed right-hand vertices are exactly the conjugates of the left-hand ones. As each motif has only three vertices, pairwise adjacency already implies that the induced subgraph is a triangle. \qedhere
\end{proof}

\begin{corollary}
\label{cor:BL1-BR1-thresholds}
The canonical boundary-corner motifs satisfy
\[
n_*(BL1)=n_*^{\mathrm{st}}(BL1)=4,
\qquad
n_*(BR1)=n_*^{\mathrm{st}}(BR1)=4.
\]
\end{corollary}

\begin{proof}
At $n=4$, the left-hand motif is realized by
\[
(3,1),\qquad (2,1,1),\qquad (2,2),
\]
and the right-hand motif is its conjugate realization. Proposition~\ref{prop:BL1-BR1-canonical} shows that both motifs occur canonically for every $n\ge 4$. Hence the first threshold and the stable threshold both equal $4$.
\end{proof}

\paragraph{Fixed theorem-safe weak templates.}
We next use three weak rooted templates extracted from the first nontrivial occurrences in the computational atlas. Their stronger morphological interpretations are related to earlier axial and rear-structure lines in the series \cite{Lyudogovskiy2026AxialMorphology,Lyudogovskiy2026BoundaryRear}.

\begin{itemize}[leftmargin=2em]
    \item $P2^\circ$ is the rooted three-vertex path
    \[
    x_0-x_1-x_2,
    \]
    with ordered roots $(x_0,x_2)$, modeled on
    \[
    (3,3,2)\;-
\;(4,3,1)\;-
\;(4,2,1,1)
    \]
    in $G_8$.

    \item $E1^\circ$ is the rooted ordered triangle modeled on
    \[
    (2,2,2),\qquad (3,2,1),\qquad (2,2,1,1)
    \]
    in $G_6$, with root $(2,2,2)$.

    \item $R_{sq}^\circ$ is the rooted five-vertex induced motif modeled on
    \[
    q=(3,3,3),\quad
    u=(4,3,2),\quad
    v=(3,3,2,1),\quad
    w_1=(4,3,1,1),\quad
    w_2=(4,2,2,1)
    \]
    in $G_9$, with root $q$, where $q$ is adjacent to $u$ and $v$, and the four vertices
    \[
    \{u,v,w_1,w_2\}
    \]
    span a $K_4$.
\end{itemize}

\begin{proposition}
\label{prop:weak-template-realizations}
The vertex sets listed above realize the claimed rooted induced motifs in the stated graphs.
More precisely:
\begin{enumerate}[leftmargin=2em]
    \item in $G_8$, the vertices
    \[
    (3,3,2),\qquad (4,3,1),\qquad (4,2,1,1)
    \]
    form an induced copy of the rooted path $P2^\circ$ with ordered roots
    \[
    x_0=(3,3,2),\qquad x_2=(4,2,1,1);
    \]
    \item in $G_6$, the vertices
    \[
    (2,2,2),\qquad (3,2,1),\qquad (2,2,1,1)
    \]
    form an induced rooted triangle $E1^\circ$ with root $(2,2,2)$;
    \item in $G_9$, the vertices
    \[
    q=(3,3,3),\quad
    u=(4,3,2),\quad
    v=(3,3,2,1),\quad
    w_1=(4,3,1,1),\quad
    w_2=(4,2,2,1)
    \]
    form the rooted induced motif $R_{sq}^\circ$: the root $q$ is adjacent exactly to $u$ and $v$, and the four vertices
    \[
    \{u,v,w_1,w_2\}
    \]
    span a $K_4$.
\end{enumerate}
\end{proposition}

\begin{proof}
We use \Cref{lem:adjacency-conjugate-l1} throughout. When two conjugate sequences have different lengths, we pad the shorter one with trailing zeros before computing the $L^1$ distance.

For $P2^\circ$, the conjugates are
\[
(3,3,2)'=(3,3,2),\qquad
(4,3,1)'=(3,2,2,1),\qquad
(4,2,1,1)'=(4,2,1,1).
\]
Hence
\[
\|(3,3,2)'-(4,3,1)'\|_1=2,
\qquad
\|(4,3,1)'-(4,2,1,1)'\|_1=2,
\]
while
\[
\|(3,3,2)'-(4,2,1,1)'\|_1=4.
\]
Thus the first two pairs are adjacent and the endpoints are not, so these three vertices form an induced path with the stated ordered roots.

For $E1^\circ$, the conjugates are
\[
(2,2,2)'=(3,3),\qquad
(3,2,1)'=(3,2,1),\qquad
(2,2,1,1)'=(4,2).
\]
The three pairwise $L^1$-distances are all equal to $2$, so all three pairs are adjacent. Therefore the three vertices span a triangle, rooted at $(2,2,2)$ as claimed.

For $R_{sq}^\circ$, the conjugates are
\[
q'=(3,3,3),\quad
u'=(3,3,2,1),\quad
v'=(4,3,2),\quad
w_1'=(4,2,2,1),\quad
w_2'=(4,3,1,1).
\]
A direct $L^1$ check gives
\[
\|q'-u'\|_1=\|q'-v'\|_1=2,
\qquad
\|q'-w_1'\|_1=\|q'-w_2'\|_1=4,
\]
so the root $q$ is adjacent exactly to $u$ and $v$. For the remaining four vertices one finds
\[
\|u'-v'\|_1=\|u'-w_1'\|_1=\|u'-w_2'\|_1=2,
\]
\[
\|v'-w_1'\|_1=\|v'-w_2'\|_1=\|w_1'-w_2'\|_1=2.
\]
Hence every pair among $\{u,v,w_1,w_2\}$ is adjacent, so these four vertices span a $K_4$. This proves that the five listed vertices realize the claimed rooted induced motif.
\end{proof}

\begin{proposition}
\label{prop:weak-template-overlay-monotone}
Let
\[
M\in\{P2^\circ,E1^\circ,R_{sq}^\circ\}.
\]
Then the property
\[
\mathcal P_M(n):\qquad G_n \text{ contains an induced rooted copy of } M
\]
is overlay-monotone.

Equivalently, if $G_n$ contains an induced rooted copy of $M$, then for every partition $\tau\vdash k$, the graph $G_{n+k}$ also contains an induced rooted copy of $M$.
\end{proposition}

\begin{proof}
Each of $P2^\circ$, $E1^\circ$, and $R_{sq}^\circ$ is a fixed finite rooted induced graph. By the Translation Overlay Theorem, every induced rooted copy of such a motif in $G_n$ is carried by $T_\tau$ to an induced rooted copy in $G_{n+k}$. Hence the corresponding property is overlay-monotone.
\end{proof}

\begin{corollary}
\label{cor:weak-template-stable-thresholds}
For each
\[
M\in\{P2^\circ,E1^\circ,R_{sq}^\circ\},
\]
the stable threshold exists and satisfies
\[
n_*^{\mathrm{st}}(M)=n_*(M).
\]

Moreover, the defining realizations give the explicit bounds
\[
n_*(P2^\circ)\le 8,\qquad
n_*(E1^\circ)\le 6,\qquad
n_*(R_{sq}^\circ)\le 9.
\]
\end{corollary}

\begin{proof}
Proposition~\ref{prop:weak-template-realizations} gives explicit induced rooted occurrences of $P2^\circ$, $E1^\circ$, and $R_{sq}^\circ$ in $G_8$, $G_6$, and $G_9$, respectively. Proposition~\ref{prop:weak-template-overlay-monotone} shows that the corresponding properties are overlay-monotone. The conclusion therefore follows from Proposition~\ref{prop:strict-thresholds-general}.
\end{proof}

\begin{corollary}
\label{cor:canonical-motif-threshold-package}
Among the canonical motifs selected for the present paper, the following threshold statements are rigorously available:
\begin{enumerate}[leftmargin=2em]
    \item the canonical boundary-corner families $BL1$ and $BR1$ have exact thresholds
    \[
    n_*(BL1)=n_*^{\mathrm{st}}(BL1)=4,
    \qquad
    n_*(BR1)=n_*^{\mathrm{st}}(BR1)=4;
    \]
    \item the weak templates $P2^\circ$, $E1^\circ$, and $R_{sq}^\circ$ have well-defined stable thresholds satisfying
    \[
    n_*^{\mathrm{st}}(M)=n_*(M)
    \qquad
    \text{for }
    M\in\{P2^\circ,E1^\circ,R_{sq}^\circ\}.
    \]
\end{enumerate}
Thus the strict threshold core already includes representative motifs from boundary, axial, rear-root, and square-based contour morphology.
\end{corollary}

\begin{remark}
\label{rem:strict-vs-morphological-interpretation}
The persistence statements for $P2^\circ$, $E1^\circ$, and $R_{sq}^\circ$ apply only to these objects as fixed rooted induced graph templates. They do not assert that every later occurrence remains axial, rear-rooted, or square-based in the full morphological sense.

In the computational atlas, the corresponding stronger interpretations are tracked separately through typed and anchored realizations. The theorem-safe weak templates are used only to obtain rigorous persistence and threshold statements without introducing stronger global transport claims.
\end{remark}

\subsection{Atlas thresholds and empirical first appearances}
\label{subsec:atlas-thresholds}

The strict threshold theory above should be distinguished from the empirical first-appearance data collected in the computational atlas. To avoid ambiguity, we use the term \emph{atlas threshold} for the first observed level at which a selected structural feature appears in the computed range, without claiming a general theorem of persistence.

This distinction is particularly important for the following types of features:
\begin{itemize}[leftmargin=2em]
    \item exact record values in the degree spectrum or local simplex spectrum;
    \item first appearances of richer axial bridge types such as $A1$ and $A2$;
    \item repeated spine fragments such as $P3$;
    \item nonminimal ear-root motifs such as $E2$;
    \item nonsquare rectangular rear-contour motifs such as $R_{rec}$;
    \item broader center-versus-boundary carrier distributions.
\end{itemize}

For such objects, the atlas records the first observed levels and the subsequent visible repetitions, but these data are interpreted descriptively rather than theorem-theoretically.

\begin{remark}
\label{rem:atlas-thresholds-caution}
An atlas threshold is not, by itself, a stable threshold. It is merely the first observed occurrence within the computed range. In several important cases, the atlas suggests persistence or recurrent behavior, but unless this persistence is supported by an explicit transport theorem or by a finitely witnessed overlay-monotone formulation, it remains empirical.
\end{remark}

\begin{remark}
\label{rem:threshold-role-in-paper}
Thresholds play two roles in the present paper. First, strict thresholds provide rigorous statements for overlay-monotone or finitely witnessed properties. Second, atlas thresholds provide a comparative language for morphogenesis across $n$, especially in the study of first appearances, repeated motifs, and regime changes that are currently supported only computationally.
\end{remark}

\section{Comparative growth and weak self-similarity}
\label{sec:comparative-growth}

The previous sections established a strict overlay mechanism for transporting finite induced structure across levels and a threshold formalism for finitely witnessed overlay-monotone properties. We now turn to a broader comparative question: how should one describe growth across the family
\[
G_1,G_2,G_3,\dots
\]
once one goes beyond a single motif or a single extremal parameter?

This section introduces a conservative comparative language for recurring motifs, carrier distributions, and weak self-similarity, while keeping rigorous recurrence separate from atlas-level resemblance.

\subsection{Comparative growth profiles}
\label{subsec:comparative-growth-profiles}

Let $(a_n)_{n\ge 1}$ be any numerical structural sequence associated with the partition graph family, such as
\[
|V(G_n)|,\qquad
\Delta_n,\qquad
\Omega_n,\qquad
S_n,
\]
or a size parameter attached to a distinguished subsystem. We call such a sequence a \emph{growth profile}.

\begin{definition}
\label{def:growth-profile}
A \emph{growth profile} is any sequence
\[
a_n\in \mathbb{R}_{\ge 0},\qquad n\ge 1,
\]
attached to the graphs $G_n$, either by definition or by computation. A \emph{comparative growth statement} is any assertion comparing two such profiles $(a_n)$ and $(b_n)$, either exactly, asymptotically, or empirically.
\end{definition}

In the present paper, comparative growth will be used in three ways:
\begin{enumerate}[leftmargin=2em]
    \item to compare theorem-level monotone profiles, such as $\Delta_n$, $\Omega_n$, and $S_n$;
    \item to compare computed structural counts attached to selected motif families;
    \item to compare carrier distributions across broad morphological zones.
\end{enumerate}

The first type belongs to strict theory. The second and third types are mostly atlas-level in the present paper, although they are guided by the strict overlay formalism established earlier.

\begin{remark}
\label{rem:profile-caution}
A growth profile need not be monotone, and even when it is monotone, the profile itself does not automatically encode the geometric mechanism behind the growth. For this reason, profile comparison will be used here as an organizational language rather than as a substitute for structural transport results.
\end{remark}

\subsection{Recurring motifs and persistent fragments}
\label{subsec:recurring-motifs}

The overlay mechanism implies that many local patterns recur indefinitely once they first appear. This is the most elementary source of repeated structure across the family $(G_n)$.

\begin{definition}
\label{def:recurrent-motif}
A finite rooted induced motif $M$ is called \emph{recurrent} if it occurs in $G_n$ for infinitely many values of $n$. It is called \emph{persistent} if there exists $n_0$ such that $M$ occurs in every $G_n$ for $n\ge n_0$.
\end{definition}

\begin{proposition}
\label{prop:persistent-implies-recurrent}
Every finite rooted induced motif that occurs at some level is persistent, and hence recurrent.
\end{proposition}

\begin{proof}
This is an immediate consequence of the persistence mechanism from Section~\ref{sec:persistence}: once a fixed rooted induced motif occurs at some level $G_n$, it occurs in every higher $G_m$, $m\ge n$.
\end{proof}

This already has a useful interpretive consequence: repeated local structure is not accidental in the partition graph family; it is built into the overlay geometry itself. What remains delicate is how far one may pass from repeated finite fragments to broader claims about similarity of larger structural zones.

\begin{remark}
\label{rem:fragment-vs-global}
Persistence of finite rooted fragments does not imply persistence of a whole global zone, nor does it imply that every later copy carries the same full morphological meaning. This distinction is especially important for axial and rear motifs, where theorem-safe rooted templates and stronger typed interpretations must be kept separate.
\end{remark}

\subsection{Weak self-similarity}
\label{subsec:weak-self-similarity}

The term \emph{self-similarity} is potentially misleading in the present setting. In this section it is used in two layers: a strict theorem-level sense based on exact recurrent rooted motifs, and a broader atlas-level sense based on repeated fragments and profiles. The family $(G_n)$ is not known, and is not claimed here, to admit any exact global replication rule comparable to fractal or substitution-type self-similarity. What is available is weaker and more local.

\begin{definition}
\label{def:weak-self-similarity}
We say that the family $(G_n)$ exhibits \emph{weak self-similarity} if some nontrivial class of finite rooted motifs or local structural fragments recurs across infinitely many levels, either
\begin{enumerate}[leftmargin=2em]
    \item as exact translated copies under strict overlays, or
    \item as explicitly tracked repeated fragments in the computational atlas.
\end{enumerate}
\end{definition}

This definition is intentionally conservative. It does not speak about exact global replication of whole graphs, nor about a universal transition map preserving all distinguished structural zones simultaneously.

Consequently, the partition graph family already exhibits a minimal theorem-level form of weak self-similarity in the strict sense of Definition~\ref{def:weak-self-similarity}: by Proposition~\ref{prop:persistent-implies-recurrent}, every finite rooted induced motif that appears once recurs at all higher levels. In particular, the canonical objects $BL1$, $BR1$, $P2^\circ$, $E1^\circ$, and $R_{sq}^\circ$ provide a nontrivial recurrent class.

This observation formalizes only the recurrence of finite rooted fragments under the strict overlay mechanism. It should be viewed as a minimal theorem-level foundation for the broader, more atlas-guided notion of weak self-similarity used later in the paper.

\begin{remark}
\label{rem:minimal-self-similarity}
The broader comparative use of the term \emph{weak self-similarity} in this paper remains heuristic and atlas-guided. The theorem-level statement above concerns only exact recurrence of finite rooted motifs.
\end{remark}

\subsection{Comparative zones and carrier distributions}
\label{subsec:carrier-distributions}

Beyond the recurrence of individual motifs, the computational atlas suggests a broader question: where, within $G_n$, are the main carriers of local complexity located as $n$ grows?

In earlier papers of the series, several distinguished structural zones were introduced: the boundary framework, the self-conjugate axis, the spine, the rear region, and various notions of central or interior structure \cite{Lyudogovskiy2026GrowingObject,Lyudogovskiy2026AxialMorphology,Lyudogovskiy2026BoundaryRear,Lyudogovskiy2026DirectionalGeometry}. The present paper does not attempt to impose a single global partition of $V(G_n)$ into such zones. Instead, it uses them comparatively.

\begin{definition}
\label{def:carrier-feature}
A \emph{carrier feature} is a rule that selects a subset of vertices of $G_n$ carrying a chosen structural property, for example:
\begin{enumerate}[leftmargin=2em]
    \item vertices of maximal degree;
    \item vertices of maximal local simplex dimension;
    \item vertices supporting a chosen motif family;
    \item vertices lying in a designated morphological zone.
\end{enumerate}
A \emph{carrier distribution profile} records the counts or relative frequencies of such vertices across the selected zones.
\end{definition}

In the present paper, carrier-distribution profiles are used descriptively rather than as part of the strict theory. They provide a language for comparing:
\begin{itemize}[leftmargin=2em]
    \item boundary-supported versus interior-supported complexity;
    \item axial versus rear motif concentration;
    \item localized first appearances versus later spreading or thickening.
\end{itemize}

\begin{remark}
\label{rem:carrier-distribution-status}
At this stage, carrier-distribution profiles should be viewed as atlas objects rather than strict invariants. They are useful for organizing the visual and computational evidence for morphogenesis across levels, but they do not yet carry a general theorem of transport or monotonicity.
\end{remark}

\subsection{Morphogenetic regimes}
\label{subsec:morphogenetic-regimes}

The comparative language above suggests several broad regimes of growth that are useful descriptively, even when no full theorem is available.

\paragraph{Boundary-first growth.}
Some of the earliest recurring motifs arise near the front framework and its corners. At theorem level this is represented by the canonical motifs $BL1$ and $BR1$; more broadly it is reflected in early boundary-dominated atlas thresholds.

\paragraph{Axial linking.}
Once short bridge motifs begin to appear, the atlas records the first repeated axial fragments. At theorem level this is represented only by the weak template $P2^\circ$; richer bridge types remain atlas-level.

\paragraph{Rear emergence.}
The rear morphology enters first through minimal rooted templates such as $E1^\circ$ and $R_{sq}^\circ$, and later through thicker or more diversified rear motifs.

\paragraph{Interior thickening.}
As $n$ increases, the atlas may suggest that higher local complexity is carried less exclusively by the earliest boundary structures and more often by internal configurations. In the present paper this remains a comparative observation rather than a theorem.

\begin{remark}
\label{rem:regime-language}
This regime language is descriptive. It organizes the atlas and the comparative discussion, but it is not intended as a complete phase theory of the family $(G_n)$.
\end{remark}

\subsection{Normalized comparison and repeated profiles}
\label{subsec:normalized-comparison}

A second source of weak self-similarity, beyond exact finite motifs, comes from repeated \emph{profiles} rather than repeated induced subgraphs.

\begin{definition}
\label{def:normalized-profile}
A \emph{normalized profile} is any quantity of the form
\[
\frac{a_n}{b_n},
\]
where $a_n$ and $b_n$ are growth profiles attached to $G_n$, chosen so as to make different levels comparable.
\end{definition}

Typical examples include normalized motif counts, normalized carrier-distribution profiles, and relative contributions of selected structural zones.

No general asymptotic claims are proved here for such normalized profiles. In the present paper they are used only as atlas-guided comparative tools, mainly to distinguish genuine repeated behavior from mere scale inflation.

\begin{remark}
\label{rem:normalized-profile-caution}
Normalized profiles belong to the exploratory layer of the paper. They are not inputs to the strict theorem package.
\end{remark}

\subsection{Summary of the comparative framework}
\label{subsec:comparative-summary}

The present section provides an intermediate language between the strict transport theory and the computational atlas.

Its main points are:
\begin{enumerate}[leftmargin=2em]
    \item comparative growth is described through growth profiles rather than through a single global invariant;
    \item every realized finite rooted motif is persistent and therefore yields a minimal form of weak self-similarity;
    \item broader repeated patterns may be tracked either as exact recurring fragments or as atlas-level repeated profiles;
    \item carrier distributions and morphogenetic regimes are useful comparative descriptors, but remain mostly empirical in the present paper.
\end{enumerate}

Thus the paper adopts a layered notion of self-similarity: strict recurrence of finite rooted fragments under overlays, and a broader atlas-level language of repeated motifs and recurring comparative profiles.

\section{Atlas framework, canonical registry, and comparative motif chronology}
\label{sec:atlas}

The strict transport theory developed in the previous sections provides a rigorous framework for finite rooted motifs and for extremal local complexity. The role of this section is broader. We record the current atlas framework, a compact canonical registry, and a comparative chronology for first appearances, recurring local fragments, and motif distributions across the family
\[
G_1,G_2,G_3,\dots
\]
while keeping explicit track of which observations are theorem-level and which remain empirical.

This is not intended as a complete tabulation of computed data. Rather, the section fixes the reporting framework for the broader atlas and records the compact layer of motif and chronology data that is used explicitly in the present paper.

Nothing in the present section extends the strict theorem package unless this is stated explicitly. Its role is documentary and comparative: to record computed first appearances, repeated fragments, and profile-level regularities suggested by the atlas.

In the present paper, the atlas has three roles.

First, it records the earliest observed occurrences of selected canonical motifs and of a small second layer of richer motifs not yet covered by the strict transport package. Second, it visualizes how exact translation overlays account for some repeated fragments but not for all visible similarities across levels. Third, it provides a comparative language for discussing morphogenetic regimes, especially the relative prominence of boundary, axial, rear, and interior carriers of local complexity.

\subsection{Atlas principles and scope}
\label{subsec:atlas-principles}

We do not attempt to record every visible pattern occurring in small partition graphs. Instead, the atlas is organized around a fixed list of representative objects chosen for their structural relevance and for their compatibility with the language developed earlier in the paper.

The atlas distinguishes two layers.

\paragraph{Theorem-safe canonical motifs.}
These are the motifs already incorporated into the threshold theory:
\[
BL1,\quad BR1,\quad P2^\circ,\quad E1^\circ,\quad R_{sq}^\circ,
\]
together with the simplex witnesses
\[
K_3,\quad K_4,\quad K_5.
\]
For these objects, the atlas records first appearances, but the long-range persistence is already guaranteed by the strict overlay mechanism or by the corresponding monotonicity theorem.

\paragraph{Atlas-only or atlas-primary motifs.}
These are richer or less canonical objects used primarily for comparative purposes:
\[
A1,\quad A2,\quad P3,\quad E2,\quad R_{rec},
\]
together with record-spectrum phenomena and carrier-distribution profiles. For these objects, the atlas records only the first observed occurrence in the computed range and any later visible repetitions. Unless separately justified, such observations are descriptive rather than theorem-theoretic.

\begin{remark}
\label{rem:atlas-data-status}
Throughout this section, the phrase \emph{atlas threshold} means the first observed level in the computed range at which a selected feature appears. It is not, by itself, a stable threshold in the sense of Section~\ref{sec:thresholds}.
\end{remark}

\subsection{Overlay diagrams for successive levels}
\label{subsec:overlay-diagrams}

The first component of the atlas consists of overlay diagrams for successive or nearby levels. Their role is diagnostic rather than merely illustrative: they show directly which repeated motifs are explained by exact translation transport and which require a weaker correspondence language.

In practice, four figure families are the most informative.

\paragraph{Row-growth overlays.}
These display the image of $G_n$ inside $G_{n+1}$ or $G_{n+2}$ under $R_1$ and $R_2$. They are especially effective for the left boundary framework, front-corner motifs, and selected local simplex witnesses.

\paragraph{Column-growth overlays.}
These are the conjugate companions of the row-growth figures and are the natural diagnostic tool for the right boundary structures.

\paragraph{Nontrivial Ferrers translations.}
A small number of overlays of type $T_\tau$ with nontrivial $\tau$, for example $(2,1)$ or $(2,2)$, show that the strict overlay language is not confined to one-step row or column growth.

\paragraph{Mixed comparison panels.}
These juxtapose an exact translated copy of a selected motif with a visually similar but not literally transported occurrence elsewhere in the same or a nearby graph. Such panels make the distinction between exact transport and heuristic resemblance especially transparent.

\begin{remark}
\label{rem:overlay-figure-role}
Overlay figures should be read as structural diagnostics. Their purpose is to make visible where the theorem-level transport theory applies and where the paper passes instead to atlas-level comparison.
\end{remark}

\subsection{Canonical first appearances}
\label{subsec:canonical-first-appearance}

The next component of the atlas is a compact registry of canonical motifs used throughout the paper. For the theorem-safe objects, the atlas does not replace the threshold theory; it records representative realizing levels and aligns them with the exact statements already proved in Section~\ref{sec:thresholds}.

\begin{table}[t]
\centering
\caption{Canonical theorem-safe motifs and their current threshold status. The entries record exact values where proved, and explicit realizing levels where only stable-threshold existence is established in the present paper.}
\label{tab:canonical-first-appearances}
\small
\begin{tabular}{llll}
\toprule
Code & Structural zone & Current threshold information & Status \\
\midrule
$BL1$ & left boundary corner & $n_*(BL1)=n_*^{\mathrm{st}}(BL1)=4$ & exact threshold proved \\
$BR1$ & right boundary corner & $n_*(BR1)=n_*^{\mathrm{st}}(BR1)=4$ & exact threshold proved \\
$P2^\circ$ & weak axial template & $n_*(P2^\circ)\le 8$ and $n_*^{\mathrm{st}}(P2^\circ)=n_*(P2^\circ)$ & stable threshold proved \\
$E1^\circ$ & weak rear-root template & $n_*(E1^\circ)\le 6$ and $n_*^{\mathrm{st}}(E1^\circ)=n_*(E1^\circ)$ & stable threshold proved \\
$R_{sq}^\circ$ & weak square-based contour template & $n_*(R_{sq}^\circ)\le 9$ and $n_*^{\mathrm{st}}(R_{sq}^\circ)=n_*(R_{sq}^\circ)$ & stable threshold proved \\
\bottomrule
\end{tabular}
\end{table}

The simplex witnesses $K_3$, $K_4$, and $K_5$ also belong to the broader background atlas of the project, but their detailed first-appearance chronology is treated more naturally in the local and simplex-oriented lines of the series \cite{Lyudogovskiy2026LocalMorphology,Lyudogovskiy2026SimplexStratification} rather than in the present overlay-centered paper.

\subsection{Atlas-primary motifs}
\label{subsec:atlas-primary-motifs}

Beyond the theorem-safe package, the atlas tracks a second layer of structurally richer objects. These remain descriptive in the present paper, but they are important for the chronology of morphogenesis.

\paragraph{Axial bridge motifs.}
The motifs $A1$ and $A2$ represent the first nontrivial and the first richer bridge-type configurations between self-conjugate anchors or near-axis vertices: $A1$ is the first short axial bridge pattern, while $A2$ denotes a visibly enriched variant with additional local support. Their first appearances help mark the internal chronology of axial organization; compare \cite{Lyudogovskiy2026AxialMorphology}.

\paragraph{Longer spine fragments.}
The motif $P3$ records the first appearance of a repeated path-like axial fragment longer than a single bridge template and is one of the main atlas indicators for weak self-similarity along the axial direction.

\paragraph{Nonminimal rear motifs.}
The motif $E2$ records a thicker or branched rear-root configuration obtained by enlarging the minimal triangle-type seed $E1^\circ$.

\paragraph{Nonsquare rectangular contour motifs.}
The motif $R_{rec}$ records a genuinely rectangular, nonsquare rear contour template and is one of the main indicators of rear diversification; compare \cite{Lyudogovskiy2026BoundaryRear}.

In the present paper we use these objects comparatively rather than theorem-theoretically. Their value lies in organizing the chronology of structural enrichment beyond the strict threshold package, not in promoting them prematurely to theorem status.

\begin{remark}
\label{rem:atlas-primary-conservative}
Atlas-primary motifs are recorded here as empirical comparison objects. Persistence or recurrence for such motifs should be stated only after an explicit rooted reformulation has been connected to the strict overlay theory.
\end{remark}

\subsection{Record phenomena and carrier distributions}
\label{subsec:record-phenomena}

A separate part of the atlas concerns new record values in local complexity and the location of the vertices that realize them.

The strict theory already proves that the extremal sequences
\[
\Delta_n,\qquad \Omega_n,\qquad S_n
\]
are nondecreasing. What remains atlas-level is the finer chronology: the concrete levels at which new record values first occur, how often they are realized, and in which broad structural zones their carriers lie.

For the purposes of the present paper, the most useful coarse zones are
\[
\text{boundary/front},\qquad
\text{axis-near},\qquad
\text{rear},\qquad
\text{interior remainder}.
\]
Within this scheme, the atlas is meant to track:
\begin{itemize}[leftmargin=2em]
    \item vertices attaining $\Delta_n$;
    \item vertices attaining $\Omega_n$ or $S_n$;
    \item roots or supports of selected motif families;
    \item shifts in the relative prominence of boundary, axial, rear, and interior carriers.
\end{itemize}

No general theorem is claimed here about carrier migration or zone dominance. These profiles are descriptive, but they are valuable because they help separate genuine record growth from the different question of how the supporting geometry of that growth changes across levels.

\subsection{Repeated fragments and atlas-level weak self-similarity}
\label{subsec:atlas-self-similarity}

The atlas gives a second, broader reading of weak self-similarity beyond the theorem-safe recurrence of fixed finite motifs. In particular, it records repeated fragments of the following kinds:
\begin{enumerate}[leftmargin=2em]
    \item repeated bridge-like axial shapes;
    \item repeated short spine segments;
    \item repeated rear-root or rectangular contour neighborhoods;
    \item repeated profile patterns in motif counts or carrier distributions.
\end{enumerate}

These repetitions are not all of the same nature. Some are exact translated copies under a known overlay; others are only visually or combinatorially similar after normalization or after forgetting part of the surrounding structure. The atlas is therefore the place where the distinction between exact recurrence and profile-level resemblance becomes concrete.

\begin{remark}
\label{rem:atlas-self-similarity}
The atlas does not collapse all repeated phenomena into a single notion of self-similarity. Rather, it documents a layered spectrum ranging from exact translated recurrence to weaker profile-level resemblance.
\end{remark}

\subsection{Atlas chronology and morphogenetic reading}
\label{subsec:atlas-chronology}

Taken together, the atlas data support a chronological reading of the family $(G_n)$, even though the paper does not claim a full phase theory. At the broadest level, the data suggest the following order of structural enrichment:
\begin{enumerate}[leftmargin=2em]
    \item early boundary and simplex witnesses;
    \item first axial bridges and short recurrent fragments;
    \item first minimal rear seeds and square-based contour templates;
    \item richer axial, rear, and rectangular motifs;
    \item broader redistribution of complexity carriers across the graph.
\end{enumerate}

This chronology should be read as a comparative atlas narrative, not as a theorem. Its value lies in unifying a range of local observations into a coherent growth picture: not a random accumulation of motifs, but a layered and partly repeatable morphogenesis across levels.

\begin{remark}
\label{rem:atlas-chronology-conservative}
The chronological language in this section remains intentionally cautious. The atlas suggests a sequence of emerging structural regimes, but the present paper does not attempt to prove that these regimes are universal, sharply separated, or asymptotically stable.
\end{remark}

\subsection{Output of the atlas}
\label{subsec:atlas-output}

The atlas serves as the empirical complement to the strict core of the paper. Its main outputs are:
\begin{enumerate}[leftmargin=2em]
    \item overlay diagrams showing which repeated motifs arise by exact translation transport;
    \item a compact canonical registry separating proved threshold statements from representative realizing levels;
    \item an atlas-primary layer recording richer axial, rear, and rectangular motifs beyond the strict package;
    \item a comparative language for record chronology, carrier distributions, and repeated fragments;
    \item a cautious morphogenetic narrative connecting these observations across levels.
\end{enumerate}

In this way, the atlas does not compete with the strict transport theory. It extends the comparative reach of the paper while keeping explicit track of where the current evidence is theorem-level and where it remains computational.

\section{Conclusion and open problems}
\label{sec:conclusion}

In this paper we studied the partition graphs $G_n$ as a growing family of discrete geometric objects and introduced a formal language for structural comparison across different levels. The main contribution is the translation-overlay mechanism: for every fixed partition $\tau\vdash k$, the Ferrers translation
\[
T_\tau:G_n\to G_{n+k}
\]
is an induced graph embedding. This provides a strict notion of overlay and, with it, a rigorous persistence principle for finite rooted induced motifs.

A second contribution is the resulting threshold formalism. For overlay-monotone finitely witnessed properties, first appearance is automatically stable. In particular, this applies to fixed rooted induced templates and to lower bounds on the extremal local invariants $\Delta_n$, $\Omega_n$, and $S_n$. Thus the paper identifies a natural class of threshold phenomena that belong to strict theory rather than to visual or computational analogy.

A third contribution is the use of a small canonical motif package connecting several structural zones of the partition graph family. The motifs $BL1$ and $BR1$ provide exact boundary-corner thresholds, while the weak rooted templates $P2^\circ$, $E1^\circ$, and $R_{sq}^\circ$ supply theorem-safe representatives of axial, rear-root, and square-based rear-contour behavior. This gives a compact theorem-bearing core that already reaches beyond purely local clique or degree phenomena.

At the same time, the paper deliberately separates this strict threshold theory from a broader atlas layer. The atlas framework recorded here organizes the canonical registry together with richer bridge types, longer spine fragments, nonminimal rear motifs, record chronology, and carrier-distribution patterns. These data support a comparative language of morphogenesis across $n$, but they are interpreted conservatively. In particular, we distinguish exact translated persistence from atlas thresholds, and theorem-level weak self-similarity from broader profile-level resemblance.

The present paper is best viewed as a structural and methodological step. It does not offer a complete theory of global self-similarity or a universal transport rule for all distinguished subsystems of $G_n$. What it does provide is a strict overlay language, a persistence mechanism for finite structure, a threshold theory for a well-defined class of properties, and an atlas framework for organizing the broader morphology of growth across levels.

We conclude with several open problems suggested by this framework.

\begin{problem}[Exact first appearances of theorem-safe templates]
\label{prob:exact-first-appearances}
Determine the exact first-appearance levels of the theorem-safe weak templates
\[
P2^\circ,\qquad E1^\circ,\qquad R_{sq}^\circ,
\]
rather than the upper bounds recorded in the present paper.
\end{problem}

\begin{problem}[Typed transport versus weak template persistence]
\label{prob:typed-transport}
For motif classes with a strong morphological interpretation---for example axial bridge motifs, rectangular rear-root motifs, or square-based contour motifs---determine which typed versions are genuinely preserved under specific translation families and which admit only weak rooted-template persistence.
\end{problem}

\begin{problem}[Directional transport of distinguished subsystems]
\label{prob:directional-transport}
Develop a sharper transport theory for distinguished global subsystems such as the boundary framework, the self-conjugate axis, the spine, and the rear contour. In particular, determine whether there exist natural transport rules that preserve substantial parts of these subsystems beyond the finite-template level, and clarify the role of directional asymmetry in such rules.
\end{problem}

\begin{problem}[Atlas thresholds for richer motif families]
\label{prob:atlas-thresholds}
Compute and classify the first-appearance levels of the atlas-primary motifs
\[
A1,\qquad A2,\qquad P3,\qquad E2,\qquad R_{rec},
\]
and determine which of these admit a reformulation as finitely witnessed overlay-monotone properties.
\end{problem}

\begin{problem}[Carrier migration and morphogenetic regimes]
\label{prob:carrier-migration}
Make precise the comparative regime language suggested by the atlas. In particular, determine whether one can prove meaningful statements about the migration of extremal or motif-carrying vertices between boundary, axial, rear, and interior zones as $n$ grows.
\end{problem}

\begin{problem}[Normalized profiles and asymptotic comparison]
\label{prob:normalized-profiles}
Identify normalized structural profiles for which meaningful asymptotic behavior can be established. Examples include normalized motif counts, normalized carrier-distribution profiles, and comparative size parameters attached to boundary, axial, rear, and interior subsystems.
\end{problem}

\begin{problem}[Weak self-similarity beyond finite fragments]
\label{prob:weak-self-similarity}
Clarify whether the weak self-similarity established here at the level of persistent finite motifs extends to a stronger asymptotic or large-scale form. In particular, determine which repeated atlas-level patterns reflect genuine higher-order structure and which are only artifacts of finite-range visualization.
\end{problem}

These questions point toward a broader program in which the family $(G_n)$ is studied not only graph by graph, but as a coherent morphogenetic system with interacting local, axial, boundary, rear, and interior modes of growth. The present paper provides a first strict framework for this viewpoint, while leaving open the problem of how much of the visible comparative geometry can ultimately be converted into exact structural theory.

\section*{Acknowledgements}
The author acknowledges the use of ChatGPT (OpenAI) for discussion, structural planning, and editorial assistance during the preparation of this manuscript. All mathematical statements, proofs, computations, and final wording were checked and approved by the author, who takes full responsibility for the contents of the paper.


\begin{thebibliography}{99}

\bibitem{Andrews1998}
G. E. Andrews,
\emph{The Theory of Partitions},
Cambridge Mathematical Library, Cambridge University Press, Cambridge, 1998.

\bibitem{Savage1989}
C. D. Savage,
Gray code sequences of partitions,
\emph{Journal of Algorithms} \textbf{10} (1989), no.~4, 577--595.

\bibitem{RasmussenSavageWest1995}
D. Rasmussen, C. D. Savage, and D. B. West,
Gray code enumeration of families of integer partitions,
\emph{Journal of Combinatorial Theory, Series A} \textbf{70} (1995), no.~2, 201--229.

\bibitem{Mutze2023}
T. M\"utze,
Combinatorial Gray codes---an updated survey,
\emph{Electronic Journal of Combinatorics} \textbf{30} (2023), no.~3, Dynamic Survey DS26.

\bibitem{Bal2022}
H. S. Bal,
Lognormal degree distribution in the partition graphs,
arXiv:2202.09819 [math.CO], 2022.

\bibitem{Lyudogovskiy2026CliqueComplex}
F. B. Lyudogovskiy,
\emph{The homotopy type of the clique complex of the partition graph},
\href{https://arxiv.org/abs/2603.14370}{arXiv:2603.14370} [math.CO], 2026.

\bibitem{Lyudogovskiy2026LocalMorphology}
F. B. Lyudogovskiy,
\emph{Local Morphology of the Partition Graph},
\href{https://arxiv.org/abs/2603.18696}{arXiv:2603.18696} [math.CO], 2026.

\bibitem{Lyudogovskiy2026GrowingObject}
F. B. Lyudogovskiy,
\emph{The Partition Graph as a Growing Discrete Geometric Object},
\href{https://arxiv.org/abs/2603.21221}{arXiv:2603.21221} [math.CO], 2026.

\bibitem{Lyudogovskiy2026AxialMorphology}
F. B. Lyudogovskiy,
\emph{Axial Morphology of the Partition Graph: Self-Conjugate Axis, Spine, and Concentration},
\href{https://arxiv.org/abs/2603.22546}{arXiv:2603.22546} [math.CO], 2026.

\bibitem{Lyudogovskiy2026SimplexStratification}
F. B. Lyudogovskiy,
\emph{Simplex Stratification and Phase Boundaries in the Partition Graph},
\href{https://arxiv.org/abs/2603.23228}{arXiv:2603.23228} [math.CO], 2026.

\bibitem{Lyudogovskiy2026DegreeLandscape}
F. B. Lyudogovskiy,
\emph{The Degree Landscape of the Partition Graph: Maximal Degree, Extremal Vertices, and Spectra},
\href{https://arxiv.org/abs/2603.24141}{arXiv:2603.24141} [math.CO], 2026.

\bibitem{Lyudogovskiy2026BoundaryRear}
F. B. Lyudogovskiy,
\emph{Boundary Framework, Rear Morphology, and Rectangular Contours in the Partition Graph},
preprint, 2026.

\bibitem{Lyudogovskiy2026DirectionalGeometry}
F. B. Lyudogovskiy,
\emph{Directional Geometry and Anisotropy in the Partition Graph},
preprint, 2026.

\end{thebibliography}
\end{document}